\documentclass[10pt]{amsart}
\usepackage{amsmath,amsfonts,latexsym,graphicx,amssymb,url, color,cite}
\usepackage{hyperref}
\newcommand{\average}{{\mathchoice {\kern1ex\vcenter{\hrule
height.4pt width 6pt depth0pt} \kern-9.7pt}
{\kern1ex\vcenter{\hrule height.4pt width 4.3pt depth0pt}
\kern-7pt} {} {} }}

\usepackage{graphicx}
\usepackage{amsfonts,delarray,amssymb,amsmath,amsthm,a4,a4wide}
\usepackage{latexsym}
\usepackage{epsfig}
\usepackage{color}
\usepackage[margin=1.12in]{geometry}

\newcommand{\R}{{\mathbb R}} 

\newcommand{\e}{\varepsilon}

\newcommand{\p}{\partial}

\renewcommand{\div}{\mbox{div}\,}

{\begin{list}{}%
         {\setlength{\leftmargin}{#1}}%
         \item[]%
}
{\end{list}}
\theoremstyle{plain}
\newtheorem{thm}{Theorem}[section]
\newtheorem{cor}[thm]{Corollary}
\newtheorem{lem}[thm]{Lemma}

\newtheorem{rem}[thm]{Remark}

\theoremstyle{definition}
\newtheorem{defn}[thm]{Definition}

\title[Optimal H\"older regularity in 2D]{On optimal H\"older regularity of solutions to the equation $\Delta u+b\cdot\nabla u=0$ in two dimensions}
\author{Nam Q. Le}
\address{Department of Mathematics, Indiana University, Bloomington, IN 47405, USA}
\email{nqle@indiana.edu}
\subjclass[2010]{35B65, 35J15, 42B30}
\keywords{Hardy space, H\"older continuity, nonlinear Hodge decomposition}

\makeatletter
\def\l@subsection{\@tocline{2}{0pt}{2.5pc}{5pc}{}}
\makeatother

\begin{document}

\begin{abstract}
 We show that for an $L^2$ drift $b$ in two dimensions, if the Hardy norm of $\text{div~}b$
 is small,  then the weak solutions to  $\Delta u+b\cdot\nabla u=0$ have the same optimal H\"older regularity as in the case of divergence-free drift, that is, $u\in C^{\alpha}_{\text{loc}}$ for all $\alpha\in (0, 1)$. 
\end{abstract}

 \maketitle

\section{Statement of the main result}

In this note, we revisit the local (optimal) H\"older continuity of $W^{1, 2}$ scalar solutions to 
\begin{equation}\Delta u + b\cdot\nabla u =0
 \label{the_eq}
\end{equation}
in two dimensions
where the drift $b= (b_1, b_2)$ is an $L^2_{\text{loc}}$ vectorfield in $\R^2$. 
When $\div b=0$, Filonov \cite[Theorem 1.2]{F} shows that
$u\in W^{2, q}_{\text{loc}}$ for all $q\in (1, 2)$ and hence the optimal H\"older regularity for $u$: $u\in C^{\alpha}_{\text{loc}}$ for all $\alpha\in (0, 1)$.  In this note, we show that the same 
conclusions hold if 
the $\mathcal{H}^{1}(\R^2)$ Hardy norm of $\text{div~}b$
 is small. The Hardy space $\mathcal{H}^{1}(\R^2)$ will be recalled in Section \ref{prelim_sec}. 
 \begin{thm}
\label{mainthm} Let $\Omega$ be an open, bounded and connected domain in $\R^2$.
Let $b\in L_{\text{loc}}^2 (\R^2; \R^2)$ and let $u\in W^{1, 2}(\Omega)$ satisfy (\ref{the_eq}) in $\Omega$ in the sense of distributions. 
 There is a small, positive constant $\e_0$ such that if $\|\div b\|_{\mathcal{H}^1(\R^2)}\leq \e_0$ 
 then $u\in W^{2, q}_{\text{loc}}(\Omega)$ for all $q\in (1, 2)$ and hence, $u\in C^{\alpha}_{\text{loc}}(\Omega)$ for all $\alpha\in (0, 1)$.
 \end{thm}

If the condition 
$\div b=0$ is dropped, then, as pointed out in \cite{F}, the solutions of (\ref{the_eq}) are not continuous nor bounded in general.
Note that,
equation (\ref{the_eq}) with the divergence-free drift $b$ appears in various models
in fluid mechanics; see, for example \cite{FV, SSSZ} and the references therein. These papers also establish several regularity results, including H\"older continuity, for solutions 
to (\ref{the_eq}) when the divergence-free drift $b$ has low integrability. In 2D,  in order to obtain the 
H\"older 
continuity of the solutions $u\in W^{1,2}$ to (\ref{the_eq}), that is, $u\in C_{\text{loc}}^{\alpha}(\Omega)$ for {\it some} $\alpha\in (0, 1)$, it suffices to assume that $\div b\in\mathcal{H}^{1}(\R^2)$. This follows from revisiting the arguments of Bethuel \cite{B}.

Denote the rotation of the gradient vector in 2D by $\nabla^{\perp} v= (-\p_2 v, \p_1 v).$
Theorem \ref{mainthm} easily gives:

\begin{cor}
\label{maincor}
Let $\Omega$ be an open, bounded and connected domain in $\R^2$. Assume that $b= h \nabla^{\perp} v$ where $h\in L^{\infty}(\Omega)\cap W^{1,2}(\Omega)$ and $v\in W^{1,2}(\Omega)$.
Let $u\in W^{1, 2}(\Omega)$ satisfy (\ref{the_eq}) in $\Omega$. 
Then
 $u\in W^{2, q}_{\text{loc}}(\Omega)$ for all $q\in (1, 2)$ and hence, $u\in C^{\alpha}_{\text{loc}}(\Omega)$ for all $\alpha\in (0, 1)$.
\end{cor}
\begin{rem} 
 The drift $b$ in Corollary \ref{maincor} satisfies $\div b \in \mathcal{H}^1$ and thus the result of Bethuel \cite{B} (see
also  \cite{HSZ} for system) already implies the 
 H\"older continuity of $u$, that is, $u\in C_{\text{loc}}^{\alpha}(\Omega)$ for some $\alpha\in (0, 1)$. The novelty of Corollary \ref{maincor} is that it 
 gives further and optimal regularity results for $u$. 
 \end{rem}
Our proof of Theorem
\ref{mainthm}
uses the Uhlenbeck-Rivi\`ere decomposition (also known as the nonlinear Hodge decomposition) and the integration by compensation to convert (\ref{the_eq}) into a {\it conservation law}. This circle of ideas was
inspired by Rivi\`ere's proof of Heinz-Hildebrandt's conjecture \cite{R1}. 
We also give another proof of Theorem \ref{mainthm} using the uniqueness result for
(\ref{the_eq}) following Filonov \cite{F}.  It is interesting to note that, {\it although (\ref{the_eq}) is linear, our arguments are non-linear.}  A key ingredient in the proof of Theorem \ref{mainthm} is the local structure of the $L^2$ vectorfields $b$ whose $\div b $ have small Hardy norm.
\begin{thm}
\label{bthm} Assume that $\Omega= B_1(0)\subset\R^2$ and $b\in L^2_{\text{loc}}(\R^2;\R^2)$. There exists a positive constant $\e_1$ with the following property.
If 
$\|b\|_{L^2({\Omega})} + \|\div b\|_{\mathcal{H}^1(\R^2)}\leq \e_1$ then
 there exist $A\in W^{1,2}(\Omega)\cap L^{\infty}(\Omega)$ and $B\in W^{1,2}(\Omega)$ such that
 $A^{-1}\in L^{\infty}(\Omega)$,  $A=1$ on $\p\Omega$, $\int_\Omega B=0$, and
 \begin{equation*}b=A^{-1}\nabla A + A^{-1}\nabla^{\perp} B.
 \end{equation*}

\end{thm}

The rest of the note is organized as follows. We recall Hardy spaces and related Wente's estimates in Sect. \ref{prelim_sec}. Assuming Theorem \ref{bthm}, we prove Theorem \ref{mainthm} in Sect. \ref{mainthm_sec}. 
We prove Theorem \ref{bthm} in Sect. \ref{bthm_sec}.

\section{Hardy spaces and related Wente's estimates}
\label{prelim_sec}
First, we recall the Hardy space $\mathcal{H}^1(\R^2)$, following H\'elein \cite[Section 3.2]{H}. For any function $f\in L^1 (\R^2)$, we denote the Riesz transforms $R_j f$ ($j=1,2$) of $f$ by
$\widehat{R_j f}(\xi)= \frac{\xi_j}{|\xi|} \hat f (\xi)~(\xi= (\xi_1,\xi_2)\in\R^2)$
where $\hat f$ is the Fourier transform of $f$:
$\hat f(\xi)=(2\pi)^{-1}\int_{\R^2} e^{-i x\cdot \xi} f(x) dx.$
\begin{defn}
 The Hardy space $\mathcal{H}^{1}(\R^2)$ with norm $\|\cdot\|_{\mathcal{H}^{1}(\R^2)}$ is defined by
 $$\mathcal{H}^{1}(\R^2)=\{f\in L^1(\R^2): \|f\|_{\mathcal{H}^{1}(\R^2)}= \|f\|_{L^1(\R^2)} +  \|R_1f\|_{L^1(\R^2)} + \|R_2 f\|_{L^1(\R^2)}<\infty\}.$$
\end{defn}
A basic observation is the following theorem, due to Coifman-Lions-Meyer-Semmes \cite{CLMS}:
\begin{thm}(\cite[Theorem 3.2.2]{H})
 \label{Hardythm}
If $u, v\in W^{1,2}(\R^2)$ then $\nabla u\cdot \nabla^{\perp} v\in \mathcal{H}^1(\R^2)$. Moreover, there is a uniform constant $C$ such that the following estimate holds:
$$\|\nabla u\cdot \nabla^{\perp} v\|_{\mathcal{H}^1(\R^2)}\leq C\|\nabla u\|_{L^2(\R^2)}\|\nabla v\|_{L^2(\R^2)}.$$
 \end{thm}
We recall the following regularity result concerning solutions to the Laplace equation with Hardy right hand side. It follows from
combining Theorems 3.3.4, 3.3.8 and together equation (3.38) in \cite{H}.
\begin{thm}
\label{HardyRHS}
Let $\Omega$ be an open subset of $\R^2$, with $C^1$ boundary. Let $f\in\mathcal{H}^1(\R^2)$, and $\phi\in L^1_{loc}(\Omega)$ be a solution of
\begin{equation*}
\Delta \phi= f~\text{in}~ \Omega,~\text{and}~
\phi  =0~\text{on}~\p\Omega.
\end{equation*}
Then $\phi\in W^{1, 2}_0(\Omega)\cap C(\overline{\Omega})$ and there is a constant depending only on $\Omega$, $C(\Omega)$, such that
$$\|\phi\|_{L^{\infty}(\Omega)} + \|\nabla \phi\|_{L^2(\Omega)}\leq C(\Omega) \|f\|_{\mathcal{H}^1(\R^2)} .$$
\end{thm}
Finally, we recall the following estimates  for boundary value problems with Jacobian structure right hand side in the theory of integration by compensation, due to Wente's \cite{W}; see also \cite{CLMS}.
\begin{lem} 
\label{HardyNeu} Let $\Omega$ be an open subset of $\R^2$, with $C^1$ boundary.
Suppose that $u, v\in W^{1,2}(\Omega)$ . Let $w$ be the unique solution in $ W^{1,p}(\Omega)$ for $1\leq p<2$ to the equation $\Delta w= \nabla u\cdot \nabla^{\perp} v~\mbox{ in}~ \Omega$, either with the Dirichlet condition
$w=0$ on $\p\Omega$, or with the Neumann boundary condition $\frac{\p w}{\p \nu}  =0$ on $\p\Omega$ and $ \int_{\Omega} w= 0$.
Then $w$ belongs to $C(\overline{\Omega})\cap W^{1,2}(\Omega)$ and there is a constant $C$ depending on $\Omega$ such that
$$\|w\|_{L^{\infty}(\Omega)} + \|\nabla w\|_{L^2(\Omega)}+  \|D^2 w\|_{L^1(\Omega)}\leq C \|\nabla u\|_{L^2(\Omega)}\|\nabla v\|_{L^2(\Omega)} .$$
\end{lem}

\section{Proofs of Theorem \ref{mainthm} and Corollary \ref{maincor}}
\label{mainthm_sec}
Assuming Theorem \ref{bthm}, we prove Theorem \ref{mainthm} and Corollary \ref{maincor} in this section. Let $\e_1$ be the positive number given by Theorem \ref{bthm}. Let $\e_0=\e_1/2$. Let $u\in W^{1, 2}(\Omega)$ satisfy (\ref{the_eq}) in $\Omega$. 
\subsection{Rescaling}
\label{res_sec}
Theorem \ref{mainthm} and Corollary \ref{maincor} are of local nature so it suffices to prove the optimal H\"older continuity of $u$ in a small ball $B_r(x_0)\subset\Omega$ round each $x_0\in\Omega$. 
We rescale the equation (\ref{the_eq}) in $B_r(x_0)$
to $B_1(0)$
by setting
$$\tilde u(x)= u(x_0 + rx), ~\tilde b(x)= r b (x_0+ rx).$$ 
Then $\tilde u\in W^{1,2}(B_1(0))$ solves
\begin{equation}-\Delta \tilde u = \tilde b\cdot \nabla \tilde u~\text{in } B_1(0).
 \label{the_eqtilde}
\end{equation}
Furthermore,
\begin{equation}\|\tilde b\|_{L^2(B_1(0))}=\|b\|_{L^2(B_r(x_0))}~\text{and}~
\|\div ~\tilde b\|_{\mathcal{H}^1(\R^2)} = \|\div b\|_{\mathcal{H}^1(\R^2)}.
\label{tbeq}
\end{equation}
By the Dominated Convergence theorem,
$\|b\|_{L^2(B_r(x_0))}\rightarrow 0~\text{as }r\rightarrow 0.$
Thus, we can fixed a radius $r>0$ so that
$\|b\|_{L^2(B_r(x_0))}<\e_0.$\\
From now on, we can assume $\Omega= B_1(0)$ with the following smallness condition on $\|\tilde b\|_{L^2(B_1(0))}$:
\begin{equation}\|\tilde b\|_{L^2(B_1(0))}<\e_0.
 \label{tbsmall}
\end{equation}
\subsection{Proofs of Theorem \ref{mainthm}}
\begin{proof}[Proof of Theorem \ref{mainthm} via conservation laws]
Suppose that $\|\div b\|_{\mathcal{H}^1(\R^2)}\leq \e_0$. 
Then from (\ref{tbeq}) and (\ref{tbsmall}), we have $$\|\tilde b\|_{L^2(B_1(0))} + \|\div ~\tilde b\|_{\mathcal{H}^1(\R^2)}\leq 2\e_0=\e_1.$$ 
Applying Theorem \ref{bthm}, we find 
$A\in W^{1,2}(B_1(0))\cap L^{\infty}(B_1(0))$ and $B\in W^{1,2}(B_1(0))$ such that
 $A^{-1}\in L^{\infty}(B_1(0))$,  $A=1$ on $\p B_1(0)$, $\int_{B_1(0)} B=0$, and
 \begin{equation*}\tilde b=A^{-1}\nabla A + A^{-1}\nabla^{\perp} B.
 \end{equation*}
This together with (\ref{the_eqtilde}) gives
$$\div (A\nabla \tilde u - B\nabla^{\perp} \tilde u)=0.$$
Thus, we have just converted (\ref{the_eq}) into a conservation law.
By \cite[Theorem 4.3]{R2} or the proof of \cite[Theorem 1.1]{R1}, $\tilde u\in W^{1, p}_{\text{loc}}(B_1(0))$ for all $p\in (1, \infty)$.  Now, the right hand side of (\ref{the_eqtilde}) belongs to $L^q_{\text{loc}} (B_1(0))$
for all $q\in (1, 2)$ and hence $\tilde u\in W^{2, q}_{\text{loc}}(B_1(0))$ for all $q\in (1, 2)$.
Therefore, by the Sobolev embedding theorem, $\tilde u\in C^{\alpha}_{\text{loc}}(B_1(0))$ for all $\alpha\in (0, 1)$. Rescaling back, we obtain
the desired regularity for $u$.
\end{proof}

\begin{proof}[Proof of Theorem \ref{mainthm} via uniqueness] 
Instead of using \cite{R1, R2}, we can give a direct and short proof of Theorem \ref{mainthm} using 
the uniqueness approach of Filonov \cite{F}.
In \cite[Theorem 1.2]{F}, Filonov proved\footnote{Our sign is different from \cite{F}. Filonov considered $-\Delta v + \tilde b\cdot \nabla v= 0$.} that $W^{1,2}$ solutions of (\ref{the_eqtilde}) belong to $ W^{2, q}_{\text{loc}}(B_1(0))$ for all $1<q<2$ provided that the equation
\begin{equation}
\Delta v + \tilde b\cdot \nabla v= 0~ \mbox{ in}\quad B_1(0),~\text{with}~
v  =0~\text{on}~\p B_1(0).
\label{homoeq}
\end{equation}
has a unique solution $v=0$ in $W^{1,2}_0(B_1(0))$. We prove that this is indeed the case in the context of Theorem \ref{mainthm}. As above, we have
$\tilde b= A^{-1} (\nabla A +\nabla^{\perp} B)$
and hence (\ref{homoeq}) becomes
\begin{equation}
\left\{\begin{array}{rl}
\div (A\nabla v) + \nabla^{\perp}B\cdot \nabla v&= 0\ \ \ \ \qquad \mbox{ in}\quad B_1(0),\\
v & =0~\ \ \ \qquad\text{on}~\p B_1(0).
\end{array}\right.
\label{homoeq2}
\end{equation}
Since $A^{-1}\in L^{\infty}(B_1(0))$, we prove that $v=0$ by showing
$\displaystyle\int_{B_1(0)} A|\nabla v|^2=0.$
Following the idea of the proof of \cite[Lemma 2.6]{F}, we choose a sequence $\psi_n\in C^{\infty}_0(B_1(0))$ such that $\psi_n\rightarrow v$ in $W^{1,2} (B_1(0))$. From
\begin{eqnarray*}\int_{B_1(0)} A|\nabla v|^2 &=&\int_{B_1(0)} A\nabla v\cdot \nabla \psi_n + \int_{B_1(0)} A\nabla v\cdot (\nabla v-\nabla \psi_n)\\ &\leq& \int_{B_1(0)} A\nabla v\cdot\nabla \psi_n
+
\|A\|_{L^{\infty}(B_1(0))}\|\nabla v\|_{L^{2}(B_1(0))}\|\nabla v-\nabla \psi_n\|_{L^{2}(B_1(0))},
\end{eqnarray*}
and $\|\nabla v-\nabla \psi_n\|_{L^{2}(B_1(0))}\rightarrow 0$ when $n\rightarrow \infty$, it remains to show that
\begin{equation}\int_{B_1(0)} A\nabla v\cdot\nabla \psi_n\rightarrow 0~\text{when}~n\rightarrow \infty.
\label{vanish1}
\end{equation}
To see this, multiplying both sides of (\ref{homoeq2}) by $\psi_n$ and using that
$$\int_{B_1(0)} \nabla^{\perp} B \cdot\nabla\psi_n \psi_n=\int_{B_1(0)} \nabla^{\perp} B \cdot\nabla(\psi^2_n/2) =0$$
which follows from integrating by parts and $\div \nabla^{\perp} B=0$,  we obtain
$$\int_{B_1(0)} A\nabla v\cdot\nabla \psi_n = \int_{B_1(0)} \nabla^{\perp} B \nabla v\psi_n =\int_{B_1(0)} \nabla^{\perp} B \cdot \nabla (v-\psi_n) \psi_n.$$
Since $\div (\nabla^{\perp}B)=0$, by \cite[Lemma 2.4]{F}, and the fact  that $\psi_n\rightarrow v$ in $W^{1,2} (B_1(0))$, we have
$$\int_{B_1(0)} \nabla^{\perp} B \cdot \nabla (v-\psi_n) \psi_n \leq C \|\nabla B\|_{L^2 (B_1(0))} \|\nabla v-\nabla \psi_n\|_{L^2 (B_1(0))} \|\nabla \psi_n\|_{L^2 (B_1(0))}\rightarrow 0 ~\text{when}~n\rightarrow \infty.$$
Therefore, we obtain (\ref{vanish1}) and the proof of Theorem \ref{mainthm} is complete.
\end{proof}
\begin{proof}[Proof of Corollary \ref{maincor}]
Suppose that $b= h \nabla^{\perp} v$ where $h\in L^{\infty}(\Omega)\cap W^{1,2}(\Omega)$ and $v\in W^{1,2}(\Omega)$. Corresponding to the rescalings of $u$ and $b$ in Section
\ref{res_sec}, we also rescale
$h$ and $v$ as follows:
$$\tilde h(x)= h (x_0 + rx)-c_1, \tilde v(x)= v(x_0 + rx)-c_2$$
where $c_1$ and $c_2$ are constants so that
$\int_{B_1(0)} \tilde h=\int_{B_1(0)} \tilde v= 0.$
Note that
$\|\nabla \tilde v\|_{L^{2}(B_1(0))}= \|\nabla v\|_{L^2 (B_r(x_0))}$
and by Poincar\'e's inequality, we have
\begin{equation}\|\tilde v\|_{W^{1,2}(B_1(0))}\leq C \|\nabla \tilde v\|_{L^{2}(B_1(0))}\leq C \|\nabla v\|_{L^2 (B_r(x_0))}.
 \label{tveq1}
\end{equation}
Similarly, we have
\begin{equation}\|\tilde h\|_{W^{1,2}(B_1(0))} \leq C \|\nabla h\|_{L^2 (B_r(x_0))}
 \label{theq1}
\end{equation}
We can extend $\tilde h$ and $\tilde v$ to be compactly supported functions
 in $\R^2$ such that
 $\tilde h\in L^{\infty}(\R^2)\cap W^{1,2}(\R^2), \tilde v\in W^{1,2}(\R^2)$, $\|\tilde h\|_{L^{\infty}(\R^2)}\leq C \|\tilde h\|_{L^{\infty}(B_1(0))},$
 and
\begin{equation}\|\tilde h\|_{W^{1, 2}(\R^2)}\leq C \|\tilde h\|_{W^{1, 2}(B_1(0))},~\|\tilde v\|_{W^{1, 2}(\R^2)}\leq C \|\tilde v\|_{W^{1, 2}(B_1(0))}.
\label{exthv} 
\end{equation}
With these extensions, we have  by Theorem \ref{Hardythm}, $\div \tilde b= \nabla \tilde h\cdot \nabla^{\perp } \tilde v\in \mathcal{H}^1(\R^2)$
 and
 $$\|\div \tilde b\|_{\mathcal{H}^1(\R^2)}\leq C\|\nabla \tilde h\|_{L^2(\R^2)}\|\nabla \tilde v\|_{L^2(\R^2)}.$$
It follows from  (\ref{tveq1}), (\ref{theq1}) and (\ref{exthv}) that
$$\|\div \tilde b\|_{\mathcal{H}^1(\R^2)}\leq C \|\nabla h\|_{L^2 (B_r(x_0))} \|\nabla v\|_{L^2 (B_r(x_0))}.$$
Using the Dominated Convergence theorem, we can now further reduce the small radius $r$ in Section \ref{res_sec} so that $\|\div \tilde b\|_{\mathcal{H}^1(\R^2)}\leq \e_0.$
Applying Theorem \ref{mainthm} to (\ref{the_eqtilde}), we obtain the conclusion of the corollary.
\end{proof}

\section{Proof of Theorem \ref{bthm}}
\label{bthm_sec}
\subsection{Uhlenbeck-Rivi\`ere decomposition } 
\label{Hodge_sec}
Recall that $\Omega= B_1(0)$ and $\div b\in \mathcal{H}^1(\R^2)$. Let $\tau= (-y, x)$, and $\nu= (x, y)$ denote the unit tangential and normal vectorfields on $\p\Omega$. 
We use the Hodge decomposition
\begin{equation}b=\nabla^{\perp} \xi -\nabla p,~\text{where}~p=0~\text{on}~\p\Omega.
 \label{Hdecomp}
\end{equation}
To do this,
let $p\in L^1_{\text{loc}}(\Omega)$ solve
$$-\Delta p= \div b~\text{in}~\Omega, ~p=0~\text{on}~\p\Omega.$$
Because $\div b \in\mathcal{H}^{1}(\R^2)$, by Theorem \ref{HardyRHS}, we have $p\in W^{1,2}(\Omega)\cap L^{\infty}(\Omega)$. Furthermore,
\begin{equation}\|p\|_{L^{\infty}(\Omega)} + \|\nabla p\|_{L^{2}(\Omega)}  \leq C \|\div b\|_{\mathcal{H}^1(\R^2)}.
 \label{pbound}
\end{equation}
With the above $p$, we have
$\div (b + \nabla p)=0$
so we can find $\xi\in W^{1,2}(\Omega)$ such that (\ref{Hdecomp}) holds. Now, suppose that we have a smallness condition on $b$, precisely, for some small $\e_1>0$ to be determined,
\begin{equation}\|b\|_{L^2({\Omega})} + \|\div b\|_{\mathcal{H}^1(\R^2)}\leq \e_1.
\label{smallbH}
\end{equation}

From (\ref{pbound}) and (\ref{Hdecomp}), we have
\begin{equation}\|p\|^2_{L^\infty(\Omega)} + \int_{\Omega}|\nabla \xi|^2 + \int_{\Omega}|\nabla p|^2 \leq C \left(\int_{\Omega} |b|^2 + \|\div b\|^2_{\mathcal{H}^1(\R^2)}\right).
 \label{Dpxi}
\end{equation}
Inspired by Rivi\`ere \cite{R1, R2}, we now rewrite (\ref{Hdecomp}) into the Uhlenbeck-Rivi\`ere decomposition (also known as the nonlinear Hodge decomposition).
Let
$P= e^{p}.$ Then (\ref{Hdecomp}) becomes a nonlinear decomposition
\begin{equation}b=\nabla^{\perp} \xi-P^{-1}\nabla P.
 \label{beq}
\end{equation}
We will use (\ref{beq}) to convert (\ref{the_eq})
into a {\it conservation law}.
Note that,
$\nabla P= e^{p}\nabla p,~\nabla P^{-1} = -e^{-p}\nabla p.$\\
Let $\e\in (0, \frac{1}{100})$ be a small constant to be chosen in Lemma \ref{ABlem} below. With this $\e$, 
we choose  $\e_1$ small so that from (\ref{smallbH}) and (\ref{Dpxi}), we have
\begin{equation}\int_{\Omega} |\nabla\xi|^2 + |\nabla P|^2 + |\nabla P^{-1}|^2<\e
 \label{DPeps}
\end{equation}
and \begin{equation}\|P\|_{L^{\infty}(\Omega)}\leq 1+\e.
     \label{Peps}
    \end{equation}
It follows that
\begin{equation}1/10\leq \|P\|_{L^{\infty}(\Omega)}, \|P^{-1}\|_{L^{\infty}(\Omega)}\leq 10.
 \label{Pbound}
\end{equation}
Recall that
$P=1~\text{on}~\p\Omega.$
To prove Theorem \ref{bthm}, it remains to prove the following lemma.
\begin{lem}
\label{ABlem} Assume that (\ref{DPeps}) and (\ref{Peps}) hold. If $\e$ is sufficiently small then
 there exist $A\in W^{1,2}(\Omega)\cap L^{\infty}(\Omega)$ and $B\in W^{1,2}(\Omega)$ such that
 $A^{-1}\in L^{\infty}(\Omega)$,  $A=1$ on $\p\Omega$, $\int_\Omega B=0$, and
 \begin{equation}Ab=\nabla A + \nabla^{\perp} B,
  \label{ABeq}
 \end{equation}
with
\begin{equation*}
 \|AP-1\|^2_{L^{\infty}(\Omega)} + \|\nabla (AP)\|^2_{L^{2}(\Omega)} + \|\nabla B\|^2_{L^{2}(\Omega)}\leq C\e.
\end{equation*}
\end{lem}
\subsection{Proof of Lemma \ref{ABlem}}
\begin{proof}[Proof of Lemma \ref{ABlem}]
Notice that, by an approximation argument using the standard mollifications, it suffices to prove the lemma for smooth vectorfields $b$. In this case,  we have on $\p\Omega$
\begin{equation}
 b\cdot\tau=(\nabla^{\perp} \xi -\nabla p)\cdot\tau= \frac{\p \xi}{\p \nu}.
 \label{xibdr}
\end{equation}
The function $\xi$ in (\ref{Hdecomp}) can be chosen to be the smooth solution 
to
$$-\Delta \xi= \text{curl } b= \p_2 b_1- \p_1 b_2~ \mbox{ in}~ \Omega,~\frac{\p \xi}{\p \nu}  =b\cdot\tau~\text{on}~\p\Omega~\text{and}~\int_{\Omega} \xi = 0.$$
In what follows, we will use equation (\ref{xibdr}).
We use a fixed point argument as in Rivi\`ere \cite{R1, R2}. Let $P$ be as in Sect. \ref{Hodge_sec}. To each $A\in W^{1,2}(\Omega)$, we associate $\tilde A= A P$. Suppose that $A$ and $B$ are solutions of
(\ref{ABeq}). Then, recalling (\ref{beq}),  and noting that $\nabla P P^{-1}=-P\nabla P^{-1}$, we have
$$\nabla \tilde A-\tilde A\nabla^{\perp}\xi + \nabla^{\perp} B.P=0~\text{and}~\nabla^{\perp}B = Ab-\nabla A= 
A \nabla^{\perp}\xi + AP\nabla P^{-1}-\nabla A.$$
Taking the divergence of the first equation and taking the curl (=$-\nabla^{\perp}$) of the second equation yield
$$\Delta \tilde A =\nabla\tilde A\cdot \nabla^{\perp} \xi- \nabla^{\perp} B\cdot\nabla P~\text{and}~\Delta B=\nabla^{\perp} (A \nabla^{\perp}\xi + AP\nabla P^{-1})= \div (A\nabla \xi) +
\nabla^{\perp}\tilde A\cdot \nabla P^{-1}.$$
We now proceed as follows.\\
{\bf Step 1.} We prove,  provided $\e$ is sufficiently small, the
existence of a solution
$(\tilde A, B)$ of the system
\begin{equation}
\displaystyle
\left\{\begin{array}{rl}
\Delta \tilde A &=\nabla\tilde A\cdot \nabla^{\perp} \xi- \nabla^{\perp} B\cdot\nabla P \ \ \ \ \qquad \quad \mbox{in}~ \Omega,\\
\Delta B  &=\div (\tilde A \nabla\xi P^{-1}) + \nabla^{\perp} \tilde A\cdot\nabla P^{-1}\ \ \ \  \mbox{in}~ \Omega,\\
\tilde A & =1~\text{and}~\frac{\p B}{\p \nu}=b\cdot \tau~\quad \quad \quad\quad\quad\quad \quad\text{on}~\p\Omega,\\
\int_{\Omega} B&= 0,
\end{array}\right.
\label{tABsys}
\end{equation}
with
\begin{equation*}
 \|\tilde A-1\|^2_{L^{\infty}(\Omega)} + \|\nabla \tilde A\|^2_{L^{2}(\Omega)} + \|\nabla B\|^2_{L^{2}(\Omega)}\leq C\e.
\end{equation*}
{\bf Step 2.} We show that (\ref{tABsys}) implies
\begin{equation}\nabla \tilde A-\tilde A\nabla^{\perp}\xi + \nabla^{\perp} B.P=0.
 \label{ABcheck}
\end{equation}
{\it Let us indicate how Steps 1 and 2  complete the proof of Lemma \ref{ABlem}.}
Assuming (\ref{ABcheck}), we find from $\tilde A= AP$ that
$P\nabla A + A\nabla P-AP\nabla^{\perp}\xi + \nabla^{\perp} B.P=0.$
Since $P$ is invertible, we obtain
$$\nabla A + A \nabla P P^{-1}-A\nabla^{\perp}\xi + \nabla^{\perp} B=0.$$
Therefore, recalling (\ref{beq}), we obtain (\ref{ABeq}). The last estimate in Lemma \ref{ABlem} follows from the last estimate in Step 1 and the fact that $\tilde A= AP$. The proof
of Lemma \ref{ABlem} is complete.
\\
{\it Proof of Step 1.}
To prove  the
existence of a solution
$(\tilde A, B)$ of (\ref{tABsys}), we will use a fixed point argument as in Rivi\`ere \cite{R1, R2}. 
Let us denote for $g\in H^{1/2}(\p\Omega)$ the space
$W^{1,2}_g(\Omega)=\{u\in W^{1,2}(\Omega), u= g~\text{on}~ \p\Omega\}.$
Consider the map 
$f(\hat A,\hat B)= (\tilde A, B)$
from
$$X= \left(W^{1,2}_1(\Omega)\cap L^{\infty}(\Omega)\right)\times W^{1,2}(\Omega)$$
into itself, where for given $(\hat A,\hat B)\in X$, the pair $(\tilde A, B)$ solves the system
\begin{equation}
\displaystyle
\left\{\begin{array}{rl}
\Delta (\tilde A-1) &=\nabla(\hat A-1)\cdot \nabla^{\perp} \xi- \nabla^{\perp}\hat B \cdot\nabla P \ \ \ \quad\quad \quad\quad\quad \quad\mbox{in}~\Omega,\\
\Delta (B-B_0)  &=\div ((\hat A-1) \nabla\xi P^{-1}) + \nabla^{\perp} (\hat A-1)\cdot\nabla P^{-1}\ \ \ \  \ \mbox{in}~ \Omega,\\
\tilde A & =1~\text{and}~\frac{\p B}{\p \nu}=b\cdot \tau~\ \quad\quad \quad\quad\quad\quad\quad\quad \quad\quad\quad \quad \text{on}~\p\Omega,\\
\int_{\Omega} B&= 0.
\end{array}\right.
\label{tABsys2}
\end{equation}
Here, the function $B_0\in W^{1, 2}(\Omega)$ 
is the solution to
 \begin{equation}
\left\{\begin{array}{rl}
\Delta B_0&= \div (\nabla\xi P^{-1})   \ \qquad \mbox{ in}~\Omega,\\
\frac{\p B_0}{\p \nu} & =b\cdot\tau~\quad\quad\quad\quad\quad\quad \text{on}~\p\Omega,\\
\int_{\Omega} B_0&= 0.
\end{array}\right.
\label{B0eq}
\end{equation}
Clearly, a fixed point of (\ref{tABsys2}) is a solution of (\ref{tABsys}).

By (\ref{xibdr}),  the equation (\ref{B0eq}) has a unique solution. Multiplying both sides of the first equation of (\ref{B0eq}) by $B_0$, 
 integrating by parts, we find from $P=1$ on $\p\Omega$  and (\ref{xibdr}) that
$\int_{\Omega} |\nabla B_0|^2= \int_{\Omega}  \nabla\xi P^{-1}\cdot \nabla B_0 $
from which we can estimate
the gradient of $B_0$ by
\begin{equation}\int_{\Omega} |\nabla B_0|^2\leq \int_{\Omega}|P^{-1}\nabla \xi|^2\leq \|P^{-1}\|^2_{L^{\infty}(\Omega)}\int_{\Omega}|\nabla \xi|^2.
 \label{gradB0}
\end{equation}
Applying Lemma \ref{HardyNeu} to $\tilde A-1$ and recalling the first and third equations in (\ref{tABsys2}), we find
\begin{equation}
 \|\tilde A-1\|^2_{L^{\infty}(\Omega)} + \|\nabla \tilde A\|^2_{L^2(\Omega)}
 \leq C \int_{\Omega} |\nabla \hat A|^2 \int_{\Omega} |\nabla \xi|^2
 + C\int_{\Omega} |\nabla\hat B|^2\int_{\Omega}|\nabla P|^2.
 \label{tAbound}
\end{equation}
Note that, by the  second and last equations in (\ref{tABsys2}), $B-B_0\in W^{1,2}(\Omega)$ satisfies
\begin{equation}
\left\{\begin{array}{rl}
\Delta (B-B_0) &= \div \left((\hat A-1) \nabla\xi P^{-1}\right) + \nabla^{\perp} \hat A\cdot\nabla P^{-1} \ \ \ \qquad \mbox{ in}\quad \Omega,\\
\frac{\p (B-B_0)}{\p \nu} & =0~\text{on}~\p\Omega,\\
\int_{\Omega} (B- B_0)&= 0.
\end{array}\right.
\label{BB0eq}
\end{equation}
At this point, we can use Theorem \ref{HardyNeu} and argue as in (\ref{gradB0}) to obtain the estimate
\begin{equation}
 \|\nabla (B-B_0)\|^2_{L^2(\Omega)}\leq C\|P^{-1}\|^2_{L^{\infty}(\Omega)}\|\hat A-1\|^2_{L^{\infty}(\Omega)}\int_{\Omega} |\nabla \xi|^2 + 
 C\int_{\Omega} |\nabla \hat A|^2 \int_{\Omega} |\nabla P^{-1}|^2.
 \label{BB0bound}
\end{equation}
This combined with (\ref{gradB0}) gives
\begin{multline}
\|\nabla B\|^2_{L^2(\Omega)}\leq C\|P^{-1}\|^2_{L^{\infty}(\Omega)}\|\hat A-1\|^2_{L^{\infty}(\Omega)}\int_{\Omega} |\nabla \xi|^2 + 
 C\int_{\Omega} |\nabla \hat A|^2 \int_{\Omega} |\nabla P^{-1}|^2 \\ +  C\|P^{-1}\|^2_{L^{\infty}(\Omega)}\int_{\Omega}|\nabla \xi|^2.
 \label{gradB}
\end{multline}
The above arguments show that for $(\hat A,\hat B), (\hat A_1,\hat B_1)\in X$, the pairs $(\tilde A, B)= f(\hat A,\hat B), 
(\tilde A_1, B_1)= f(\hat A_1,\hat B_1)$ satisfy the estimates
\begin{multline*}
 \|\tilde A-\tilde A_1\|^2_{L^{\infty}(\Omega)} + \|\nabla (\tilde A-\tilde A_1)\|^2_{L^2(\Omega)}
 \leq C \int_{\Omega} |\nabla (\hat A-\hat A_1)|^2 \int_{\Omega} |\nabla \xi|^2
 + C\int_{\Omega} |\nabla (\hat B-\hat B_1)|^2\int_{\Omega}|\nabla P|^2.
\end{multline*}
and
\begin{multline*}
\|\nabla (B-B_1)\|^2_{L^2(\Omega)}\leq C\|P^{-1}\|^2_{L^{\infty}(\Omega)}\|(\hat A-\hat A_1)\|^2_{L^{\infty}(\Omega)}\int_{\Omega} |\nabla \xi|^2 + 
 C\int_{\Omega} |\nabla (\hat A-\hat A_1)|^2 \int_{\Omega} |\nabla P^{-1}|^2.
\end{multline*}
Since $\int_{\Omega} (B-B_1)= 0$, we have by the Poincar\'e inequality
$\int_{\Omega} |B-B_1|^2\leq C\int_{\Omega}|\nabla (B-B_1)|^2.$
Hence, if
$\int_{\Omega} |\nabla\xi|^2 + |\nabla P|^2 + |\nabla P^{-1}|^2\leq \e$
is sufficiently small, then a standard fixed point argument in the 
space $X=\left(L^{\infty}(\Omega)\cap W_1^{1,2}(\Omega)\right)\times W^{1,2} (\Omega)$ yields the existence of a solution
$(\tilde A, B)$ of the system (\ref{tABsys}).

Furthermore, 
from (\ref{tAbound}), (\ref{gradB}) together with (\ref{DPeps}) and (\ref{Pbound}), the solution
$(\tilde A, B)$ of the system (\ref{tABsys}) satisfies
\begin{eqnarray*}\|\tilde A-1\|^2_{L^{\infty}(\Omega)} + \|\nabla \tilde A\|^2_{L^2(\Omega)} &\leq& C\e \|\nabla \tilde A\|^2_{L^2(\Omega)} + C\e |\nabla B|^2_{L^2(\Omega)}
 \\ &\leq & C\e \|\nabla \tilde A\|^2_{L^2(\Omega)} + C\e^2 (1 + \|\tilde A-1\|^2_{L^{\infty}(\Omega)} + \|\nabla \tilde A\|^2_{L^2(\Omega)}).
\end{eqnarray*}
Thus, if $C\e\leq 1/4$, we have
$$\|\tilde A-1\|^2_{L^{\infty}(\Omega)} + \|\nabla \tilde A\|^2_{L^2(\Omega)} \leq C\e.$$
By (\ref{gradB}), we then have
$ \|\nabla B\|^2_{L^2(\Omega)} \leq C\e.$
The proof of Step 1 is complete.\\
{\it Proof of Step 2.}
To show (\ref{ABcheck}), we introduce the Hodge decomposition
\begin{equation}\nabla \tilde A-\tilde A\nabla^{\perp}\xi + \nabla^{\perp} B.P= \nabla E + \nabla^{\perp} D
 \label{Hdecomp2}
\end{equation}
where $E=0$ on $\p\Omega$. Taking divergence of both sides of (\ref{Hdecomp2}), and recalling the first equation of (\ref{tABsys}), we get $\Delta E=0$ in $\Omega$. 
Hence, $E\equiv 0$ and (\ref{Hdecomp2})
reduces to
\begin{equation}\nabla \tilde A-\tilde A\nabla^{\perp}\xi + \nabla^{\perp} B.P= \nabla^{\perp} D.
 \label{Deq}
\end{equation}
It remains to show that $D$ is a constant. 
By (\ref{Deq}), and recalling (\ref{tABsys}), we have
$$\nabla D=-\nabla ^{\perp} \tilde A-\tilde A\nabla \xi+ P\nabla B~\text{and}~\Delta D= P(\tilde A\nabla\xi\cdot \nabla P^{-1} + \nabla^{\perp} \tilde A\cdot\nabla P^{-1} + \nabla B\cdot\nabla P P^{-1}).$$
It follows from a simple calculation that
\begin{equation}\div(\nabla D P^{-1}) = (-\nabla ^{\perp} \tilde A-\tilde A\nabla \xi+P \nabla B)\cdot\nabla P^{-1} + \Delta D P^{-1} = \nabla B \cdot \nabla (P P^{-1})=0.
\label{Diden}
\end{equation}
With (\ref{Diden}), we complete the proof of Step 2 as follows. Taking dot product with $\tau$ on both sides of (\ref{Deq}), and recalling that $P=\tilde A=1$ on
$\p\Omega$, we find that on $\p\Omega
=\p B_1(0)$:
$$D_{\nu}=\nabla^{\perp} D\cdot\tau= \tilde A_{\tau}-\tilde A\xi_\nu + B_\nu P=B_\nu-\xi_\nu=b\cdot\tau-b\cdot\tau =0.$$
Multiplying both sides of (\ref{Diden}) by $D$ and integrating by parts, we find
$$0=\int_\Omega \div (\nabla D\cdot P^{-1}) D= -\int_{\Omega} P^{-1}|\nabla D|^2 + \int_{\p\Omega} P^{-1}D \nabla D\cdot\nu = -\int_{\Omega} P^{-1}|\nabla D|^2.$$
Recalling (\ref{Pbound}), we obtain $\nabla D=0$ in $\Omega$ and hence $D$ is a constant in $\Omega$.
\end{proof}

\end{document}